\documentclass[reqno]{amsart}
\usepackage{amsmath}
\usepackage{amssymb, hyperref}

\usepackage{graphicx,color}
\usepackage{slashed}
\renewcommand{\Im}{\operatorname{Im}}
\usepackage{mathtools}
\mathtoolsset{showonlyrefs}
\newcommand{\R}{\mathbb{R}}

\usepackage{parskip}

\usepackage{xcolor}

\begin{document}
\title[On the IBNLS equation]{Scattering for the non-radial inhomogenous biharmonic NLS equation}
	
     \author[L. CAMPOS ]
	{LUCCAS CAMPOS}  
	
	\address{LUCCAS CAMPOS  \hfill\break
	Universidade Federal de Minas Gerais, MG, Brazil.}
	\email{luccasccampos@gmail.com}

	\author[C. M. GUZM\'AN ]
	{CARLOS M. GUZM\'AN } 
	
	\address{CARLOS M. GUZM\'AN \hfill\break
		Department of Mathematics, Fluminense Federal University, BRAZIL}
	\email{carlos.guz.j@gmail.com}

\begin{abstract}
We consider the focusing inhomogeneous biharmonic nonlinear Schr\"odinger equation in $H^2(\mathbb{R}^N)$,
\begin{equation*}
iu_t + \Delta^2 u - |x|^{-b}|u|^{\alpha}u=0,\\
\end{equation*}
when $b > 0$ and $N \geq 5$. We first obtain a small data global result in $H^2$, which, in the five-dimensional case, improves a previous result from  Pastor and the second author. In the sequel, we show the main result, scattering below the mass-energy threshold in the intercritical case, that is, $\frac{8-2b}{N} < \alpha <\frac{8-2b}{N-4}$, without assuming radiality of the initial data. The proof combines the decay of the nonlinearity with Virial-Morawetz-type estimates to avoid the radial assumption, allowing for a much simpler proof than the Kenig-Merle roadmap.

\ 

\end{abstract}


	\maketitle  
	\numberwithin{equation}{section}
	\newtheorem{theorem}{Theorem}[section]
	\newtheorem{proposition}[theorem]{Proposition}
	\newtheorem{lemma}[theorem]{Lemma}
	\newtheorem{corollary}[theorem]{Corollary}
	\newtheorem{remark}[theorem]{Remark}
	\newtheorem{definition}[theorem]{Definition}
	%%%%%%%%%%%%%%%%%%%%%%%%%%%%%%%%%%%%%%%%%%%%%%%%%%%%%%%%%%%%%%%%%%%%%
	%%%%%%%%%%%%%%%%%%%%%%%%%%%%%%%%%%%%%%%%%%%%%%%%%%%%%%%%%%%%%%%%%%%%%

\section{Introduction}
\indent In this paper, we study the Cauchy problem for the focusing inhomogeneous biharmonic nonlinear Schrö\-din\-ger equation (IBNLS for short)
\begin{equation}\label{IBNLS}
\begin{cases}
iu_t + \Delta^2 u - |x|^{-b}|u|^{\alpha}u=0,\\
u(0) = u_0 \in H^2(\mathbb{R}^N), 
\end{cases}
\end{equation}
where, $N \geq 5$ and $\alpha,b>0$. The limiting case b = 0 (classical biharmonic nonlinear Schr\"odinger equation (BNLS)) was introduced by Karpman \cite{karpman1996} and Karpman-Shagalov \cite{karpman-Shagalov} in order to take into consideration the role of small fourth-order dispersion terms in the propagation of intense laser beams in a bulk medium with Kerr nonlinearity.
%\begin{equation}\label{cond_p}
%\frac{8-2b}{N} < \alpha < \frac{8-2b}{N-4}.
%\end{equation}

The IBNLS equation is invariant under the scaling, $u_\mu(t,x)=\mu^{\frac{4-b}{\alpha}}u(\mu^4 t,\mu x)$,  $\mu >0$. This means if $u$  is a solution of \eqref{IBNLS}, with initial data $u_0$, so is $u_\mu$ with initial data $u_{\mu,0}=\mu^{\frac{4-b}{\alpha}}u_0(\mu x)$.  A straightforward computation yields
$$
\|u_{0,\mu}\|_{\dot{H}^s}=\mu^{s-\frac{N}{2}+\frac{4-b}{\alpha}}\|u_0\|_{\dot{H}^s},
$$
implying that the scale-invariant Sobolev space is $\dot{H}^{s_c}(\mathbb{R}^N)$, with $s_c=\frac{N}{2}-\frac{4-b}{\alpha}$, the so called \textit{critical Sobolev index}. 

When $s_c = 0$ (or $\alpha=\frac{8-2b}{N}$), the critical space is $L^2$, which is naturally associated to the conserved
mass of solutions, defined by
\begin{equation}%\label{mass}
M[u(t)]=\int_{\mathbb{R}^N}|u(t,x)|^2dx.
\end{equation}
On the other hand, when $s_c = 2$ (or $\alpha=\frac{8-2b}{N-4}$), the critical space is $\dot{H}^2$, which is naturally associated to the conserved energy of solutions, defined by
\begin{equation}\label{energy}
E[u(t)]=\frac{1}{2}\int_{\mathbb{R}^N}| \Delta u(t,x)|^2dx-
\frac{1}{\alpha +2} \int_{\mathbb{R}^N}|x|^{-b}|u|^{\alpha +2}dx.
\end{equation}
Here, we are interested in studying the $L^2$-supercritical and $\dot{H}^2$-subcritical case (also called intercritical), i.e., $0<s_c<2$ (or 
$\frac{8-2b}{N} < \alpha <\frac{8-2b}{N-4}$).

%The homogeneous case $b = 0$ is known as the biharmonic nonlinear Schr\"odinger (BNLS) equation, which has been the subject of a great deal of recent
%mathematical interest (see, for instance, the works of Bourgain \cite{Bo99}, Cazenave \cite{cazenave}, Linares-Ponce \cite{LiPo15} and Tao \cite{TaoBook}).

First, we briefly review some recent developments for the IBNLS model. It was first studied by Cho-Ozawa-Wang \cite{Cho-Ozawa}, they considered the inhomogeneous power type $|x|^{-2}|u|^{\frac{4}{N}} u$ and showed the existence of weak solutions by regularizing the nonlinearity. Recently Pastor and the second author in \cite{GUZPAS}, using the Strichartz estimates, considered the more general power-like nonlinearities of the form $|x|^{b}|u|^\alpha u$. They obtained local well-posedness in $H^2$ for $N\geq 3$, $0< b <\min\{\tfrac{N}{2},4\}$ and $\min\{\tfrac{2(1-b)}{N}, 0\}<\alpha <4^*$, where $4^*=\frac{8-2b}{N-4}$ if $N\geq 5$ or $4^*=\infty$ if $N=3,4$. Moreover, in the mass-supercritical and energy-subcritical cases, $\tfrac{8-2b}{N}<\alpha<4^*$, they showed the small data global existence for dimensions $N\geq 3$, with extra assumptions on the parameters $\alpha$ and $b$ when $N=5,6,7$, that is, they did not show global solution in the full intercritical regime of $\alpha$, and they assume an extra upper bound on $b$. More recently, for the local theory in $H^2$, the lower bound ($\alpha>\frac{2(2-b)}{N}$) was removed by Liu-Zhang \cite{Xuan-Ting} using the Besov spaces. They also established local well-posedness in $H^s$ with $0<s\leq 2$. On the other hand, for the global theory, the authors in \cite{GuzPas1} improved the results showed in \cite{GUZPAS}, by making use of the Hardy inequality, removing the extra assumptions for $N=6,7$. However for the dimension $N=5$ the authors did not cover the full range on $b$ where the local solution was obtained. More precisely, they showed global well-posedness for $\frac{8-2b}{7}<\alpha<7-2b$ and $0<b<\frac{5}{2}$ or $\frac{8-2b}{5}<\alpha <8-2b$ and $0<b<\frac{3}{2}$. The gap $\frac{3}{2}\leq b<\frac{5}{2}$ was still an open problem, which we solve here. Thus, our first goal is to improve the global result in the intercritical $5D$ to $0<b<\frac{5}{2}$.
\begin{theorem}\label{GWPH2}
Assume $N\geq 5$, $0<b<\min \left\{\frac{N}{2},4\right\}$ and $\frac{8-2b}{N}<\alpha<\frac{8-2b}{N-4}$. If $u_0 \in H^2$ satisfies $\|u_0\|_{H^2}\leq E$, for some $E>0$, then there exists $\delta_{sd}=\delta_{sd}(E)>0$ such that if $\|e^{it\Delta^2}u_0\|_{B(\dot{H}^{s_c})}<\delta_{sd}$, then there exists a unique global solution $u$ of \eqref{IBNLS} such that
\begin{equation*}\label{NGWP3}
\|u\|_{B(\dot{H}^{s_c})}\leq  2\|e^{it\Delta^2}u_0\|_{B(\dot{H}^{s_c})}\qquad \textnormal{and}\qquad \|u\|_{B\left(L^2\right)}+\|\Delta  u\|_{B\left(L^2\right)}\leq 2c\|u_0\|_{H^2}.
\end{equation*}
for some universal constant $c>0$. %Moreover, the solution $u$ scatters in both time directions in $H^2(\mathbb{R}^N)$.
\end{theorem}

To this end, we rely on the contraction mapping principle combined with the Strichartz estimates. The singular factor $|x|^{-b}$ brings some extra difficulties and the major problem is to estimate $\nabla (|x|^{-b}|u|^\alpha u)$. Here, by making use only of Sobolev and H\"older inequalities (therefore do not using Hardy inequality as in \cite{GuzPas1}), we were able to perform suitable nonlinear estimates, in carefully chosen Sobolev spaces (see Lemma \ref{L:NL} above), in order to estimate the inhomogeneous term in the whole desired range of parameters.     

Once the global result is established, the natural problem to study is the asymptotic behavior of such global solution as $t \rightarrow \infty$. The main goal in this paper is to show scattering of \eqref{IBNLS} in $H^{2}$ below the ground state. Before stating the result,
we recall the recently work showed in \cite{carguzpas2020}, which established sufficient conditions for global existence of solutions in $\dot H^{s_c}\cap H^2$ with $0\leq s_c < 2$. We state below this result for the intercritical case, that is, $0<s_c<2$ since this is the case in which we are interested.

\begin{theorem}\label{GWP} Assume $N\geq 3$, $\frac{8-2b}{N}<\alpha< 4^*$ and $0<b<\min\{\tfrac{N}{2},4\}$. Suppose $u_0 \in H^2$ obeys
\begin{equation}\label{cond1}
E[u_0]^{s_c}M[u_0]^{2-s_c}<E[Q]^{s_c}M[Q]^{2-s_c}
\end{equation}
and
\begin{equation}\label{cond2}
\|\Delta u_0\|_{L^2}^{s_c}\|u_0\|_{L^2}^{2-s_c}<\|\Delta Q\|_{L^2}^{s_c}\|Q\|_{L^2}^{2-s_c}.
\end{equation}
Then the corresponding solution $u$ is a global solution in $H^2$. In addition,
\begin{equation}\label{cond4}
\|\Delta u(t)\|_{L^2}^{s_c}\|u(t)\|_{L^2}^{2-s_c}<\|\Delta Q\|_{L^2}^{s_c}\|Q\|_{L^2}^{2-s_c}.
\end{equation}
Here $Q$ denotes the ground state solution to
%\[
$\Delta^2 Q - Q - |x|^{-b}Q^{\alpha+1}=0.$
%\]
\end{theorem}

Here, we prove, for either radial or non-radial initial data, that the global solution obtained in the above theorem also scatters in $H^2$ for $N\geq 5$ without imposing any extra restriction on the parameters $\alpha$ and $b$. The method of proof is based on the ideas developed in \cite{MurphyNonradial}, which established scattering for the $3D$ cubic inhomogeneous NLS equation\footnote{Inhomogeneous NLS model: $iu_t + \Delta u + |x|^{-b}|u|^{\alpha}u=0$.} assuming non-radial data (see also \cite{CardosoCampos} for higher dimensions). It avoids the concentration-compactness and rigidity technique developed by Kenig and Merle \cite{KENIG}. Note that we do not have the Galilean transformation for the IBLNS, nor conservation of momentum, which makes the concentration-compactness approach even harder if one wants to consider non-radial data. Our main result is the following.

\begin{theorem}\label{Scattering}
Assume $N\geq 5$, $0<b<\min\{\frac{N}{2},4\}$, and $\frac{8-2b}{N}<\alpha<\frac{8-2b}{N-4}$. If \eqref{cond1} and \eqref{cond2} hold, then the corresponding solution $u$ to \eqref{IBNLS} scatters in $H^2$ both forward and backward in time. That is, there exist $\phi^{\pm}$ such that
$$
\lim_{t\rightarrow \;\pm \infty}\|u(t)-e^{it\Delta^2}\phi^{\pm}\|_{H^2}=0.
$$
\end{theorem}

In the particular case $b=0$ (BNLS), the scattering result was showed by several authors. We mention two works that considered the $L^2$-supercritical and $\dot{H}^2$-subcritical case. Using the concentration-compactness and
rigidity argument, Guo in \cite{Guo}, for $N\geq 1$, showed scattering for the radial initial data and very recently, Dinh in \cite{DinhRadialScattersIBNLS} gave an alternative proof using the ideas of Dodson-Murphy \cite{Dod-Mur}, for $N\geq 2$. The non-radial setting is still an open problem for $b=0$. 
%However, when $b\neq 0$ it is possible to show that the global solution scatters in $H^2$. 

Theorem \ref{Scattering} extends the result showed by Saanouni \cite{SaanouniRadial} from the radial to the non-radial setting, which also assumed an extra lower bound on $\alpha$, $\alpha>\{\frac{8-2b}{N},x_0\}$, where $x_0$ is the positive root of $(2x -1)(x-1)-\frac{4-b}{N-4}$. However, this proof is questionable since the estimate used in the proof of small data scattering in \cite{SaanouniRadial} (after Proposition $3.8$)
$$
\|u(t)-e^{it\Delta^2}\phi^{\pm}\|_{H^2}\lesssim \|u\|_{L^\infty_{[t,\infty)} H^2}\|u\|^{\alpha+1}_{L^a_{[t,\infty)}L^r_x},
$$
does not
%\footnote{The pair $(a,r)$ is $H^s$-biharmonic admissible, see Section 2.} 
seem clear to us. In addition, in dimension $N=5$, they assumed $\alpha<8-2b$, but as was mentioned above, in the previous works \cite{GUZPAS, GuzPas1}, it was necessary to assume an extra hypothesis on $\alpha$ or $b$ to estimate the derivative of  $|x|^{-b}|u|^{\alpha }u$, and in \cite{SaanouniRadial}, they did not provide a new one without these restrictions. Here, we removed the extra conditions, and thus our result fits into the broader context of sharp scattering thresholds for the intercritical inhomogeneous biharmonic NLS equation.
%use the Strichartz estimates together with more careful estimates on the nonlinearity\footnote{It is worth mentioning that the assumptions on $\alpha$ in Theorem \ref{Scattering} are the same ones that appear in the estimates on the nonlinearity. In particular, we drop the technical assumption previously used in dimension $N=5$, thus covering the full intercritical case.} to prove scattering (see Lemmas \ref{small_data} and \ref{bound-scattering}) and thus extend the result of \cite{SaanouniRadial} from the radial to the non-radial setting, also removing an extra technical assumption on $\alpha$ (the lower bound). 
%To do that, we use the recent ideas introduced in \cite{MurphyNonradial} and \cite{CardosoCampos} to show scattering for the inhomogeneous NLS model. 
The argument exploits the decay of the nonlinearity together with Virial-Morawetz-type estimates. It gives an simple proof for the energy scattering that completely avoids the use of the concentration-compactness and
rigidity argument. 

\begin{remark}
The restriction to $N \geq 5$ comes from relying on the decay of the linear biharmonic operator, which, by Strichartz estimates, can only reach up to $|t|^{-N/4}$, therefore not being integrable on $[1,+\infty)$ when $N \leq 4$. The argument can be extended to radial data in $N \geq 3$, by using the Strauss Lemma to get a faster decay in the Virial-Morawetz estimate \eqref{virial-morawetz-decay}. Note that, since we are only converting the space decay of the inhomogeneity $|x|^{-b}$ in time decay, it can be very slow if $b$ is close to zero. Scattering for non-radial data in lower dimensions still remains an open problem.
\end{remark}

We finally state the scattering criterion that will be used to show Theorem \ref{Scattering}. It was first proved for the radial $3D$ cubic NLS equation by Tao \cite{Tao}, and here we prove this result in a non-radial setting for the IBNLS. 
\begin{proposition}[Scattering criterion]\label{scattering_criterion}
Assume the same hypotheses as in Theorem \ref{Scattering}. Consider an $H^2(\mathbb{R}^N)$-solution $u$ to \eqref{IBNLS} defined on $[0,+\infty)$ and assume the a priori bound 
\begin{equation}\label{E}
\displaystyle\sup_{t \in [0,+\infty)}\left\|u(t)\right\|_{H^2_x} := E < +\infty.
\end{equation}

There exist constants $R > 0$ and $\epsilon>0$ depending only on $E$, $N$, $\alpha$ and $b$ (but never on $u$ or $t$) such that if
\begin{equation}\label{scacri}
\liminf_{t \rightarrow +\infty}\int_{B(0,R)}|u(x,t)|^2 \, dx \leq \epsilon^2,
\end{equation}
then $u$ scatters forward in time in $H^2(\mathbb{R}^N)$.
\end{proposition}

Note that this result shows that general solutions to the IBNLS under the ground state behave similarly to radial solutions, in the sense that it is enough to show that the mass of the solution escapes from a (possibly large) ball centered at the origin to prove scattering. In comparison to the concentration-compactness and rigidity approach, this means that even when $b$ is very close to zero, the inhomogeneous factor $|x|^{-b}$ automatically precludes the existence of an unbounded translation parameter in the compactness part, even though the Galilean boost is unavailable for this equation. This last claim will be clarified in a future work, for which the framework will be particularly useful to handle the non-radial case in lower dimensions.

The rest of the paper is organized as follows. In section $2$, we introduce some notations, give a review of the Strichartz estimates and prove Theorem \ref{GWPH2}. In Section $3$ we obatin the scattering criterion and finally, in Section $4$, we show the main result (Theorem \ref{Scattering}).

\ 

\section{\bf Notation and Preliminaries}\label{sec2}

Let us start this section by introducing the notation used throughout the paper. We write $a\lesssim b$ to denote $a\leq cb$ for some $c>0$, denoting dependence on various parameters with subscripts when necessary. 
%If $a\lesssim b\lesssim a$, we write $a\sim b$. %We also use the notation $a{\pm}$ to denote $a\pm\eps$ for sufficiently small $\eps>0$. 

We make use of the standard Lebesgue spaces $L^p$, the mixed Lebesgue spaces $L_t^q L_x^r$, as well as the homogeneous and inhomogeneous Sobolev spaces $\dot H^{s,r}$ and $H^{s,r}$.  When $r=2$, we write $\dot H^{s,2}=\dot H^s$ and $H^{s,2}=H^s$.  If necessary, we use subscripts to specify which variable we are concerned with. We use $'$ to denote the H\"older dual. 

\subsection{Well-posedness theory}

To discuss the well-posedness theory for \eqref{IBNLS}, we first recall the Strichartz estimates in the form that we will need them. We say the pair $(q, r)$ is biharmonic Schr\"odinger admissible (B-admissible for short) if it satisfies
$$
\frac{4}{q}+\frac{N}{r}=\frac{N}{2},
$$
where 
\begin{equation}\label{L2Admissivel}
\begin{cases}
2\leq  r  < \frac{2N}{N-4},\hspace{0.5cm}\textnormal{if}\;\;\;  N\geq 5,\\
2 \leq  r < + \infty,\;  \hspace{0.5cm}\textnormal{if}\;\;\; N=4\\
2 \leq  r \leq + \infty,\;  \hspace{0.5cm}\textnormal{if}\;\;\;1\leq N\leq 3.
\end{cases}
\end{equation}

Given a real number $s<2$, we also called the pair $(q, r)$ is $H^s$-biharmonic admissible if  
$$
\frac{4}{q}+\frac{N}{r}=\frac{N}{2}-s,
$$ 
%where $\frac{10}{5-2s}\leq r<\frac{2N}{N-4}$. 
with
\begin{equation}\label{HsAdmissivel}
\begin{cases}
\frac{2N}{N-2s} \leq  r  <\frac{2N}{N-4}\;\;\;\textnormal{if}\;\; \; N\geq 5,\\
\frac{2N}{N-2s} \leq  r < + \infty\;\;\;  \textnormal{if}\;\;\;1\leq N\leq 4.
\end{cases}
\end{equation}

Given $s\in\R$, we define $\mathcal{A}_s$ to be the set of $\dot H^{s}$-biharmonic admissible pairs and introduce the Strichartz norm
\[
\|u\|_{B(\dot H^s,I)} = \sup_{(q,r)\in\mathcal{B}_s}\|u\|_{L_I^q L_x^r}.
\]
In the same way, the dual Strichartz norm is given by
\[
\|u\|_{B'(\dot H^{-s},I)} = \inf_{(q,r)\in\mathcal{B}_{-s}}\|u\|_{L_I^{q'} L_x^{r'}}.
\]
If $s=0$ then $B_0$ is the set of all B-admissible pairs. Thus, $\|u\|_{B(L^2,I)}=\sup_{(q,r)\in B_0}\|u\|_{L^q_IL^r_x}$ and $\|u\|_{B'(L^2,I)}=\inf_{(q,r)\in B_0}\|u\|_{L^{q'}_IL^{r'}_x}$. We also define the following norm
\begin{equation}\label{norm-H2}
    \|\langle \Delta \rangle u \|_{B(L^2,I)}=\|u\|_{B(L^2,I)}+\|\Delta u\|_{B(L^2,I)}.
\end{equation} 
When the $x$-integration is restricted to a subset $A \subset \R^N$ the mixed norm will be denoted by $\|f\|_{L^q_IL^r(A)}$. If $I=\mathbb{R}$, we often omit $I$.   

We now recall the Strichartz estimates for the fourth-order Schrödinger equation, which are the main tools for studying \eqref{IBNLS}, especially for well-posedness and scattering. The last one is the Strichartz estimate with a gain of derivative. See for instance, \cite{Pausader07} and \cite{GUZPAS} (see also \cite{Guo}). 
\begin{align}\label{SE1}
\qquad \quad\|e^{it\Delta^2}f\|_{B(\dot H^s,I)}\;\; \lesssim  \|f\|_{\dot H^{s}}
\end{align}
\begin{align}\label{SE2} \noeqref{SE2}
\biggl\| \int_0^t e^{i(t-t')\Delta^2} g(t')\,dt'\biggr\|_{B(\dot H^s,I)} &\lesssim \|g\|_{B'(\dot H^{-s},I)}. 
\end{align}
\begin{align}\label{SE3} \noeqref{SE3}
\biggl\|\Delta \int_0^t e^{i(t-t')\Delta^2} g(t')\,dt'\biggr\|_{B(L^2,I)} &\lesssim \|\nabla g\|_{L^2_IL^{\frac{2N}{N+2}}_x}. 
\end{align}
\begin{remark}
We also obtain (local-in-time estimate)
\begin{align}\label{SE4}
\biggl\| \int_a^b e^{i(t-t')\Delta^2} g(t')\,dt'\biggr\|_{B(\dot H^s,\R)} &\lesssim \|g\|_{B'(\dot H^{-s},[a,b])}. 
\end{align}
\end{remark}

Now, we turn our attention to proof Theorem \ref{GWPH2}. As usual the core of the proof is to establish good estimates on the nonlinearity $F(x,u)=|x|^{-b}|u|^\alpha u$. The next lemma provides these estimates, which also play an important role in proving the main result.

%The key to proving the results mentioned above is to establish good estimates on the nonlinearity.  
\begin{lemma}[Nonlinear estimates]\label{L:NL} Let $N\geq 5$, $0 < b < \min\{\frac{N}{2},4\}$ and $\frac{8-2b}{N}<\alpha<\frac{8-2b}{N-4}$. There exist positive parameters $\theta \ll \alpha$ and $\alpha_1 < \alpha - \theta$ such that
\begin{itemize}
\item [(i)] $\left \||x|^{-b}|u|^\alpha v \right\|_{B'(\dot{H}^{-s_c},I)} \lesssim \| u\|^{\theta}_{L^\infty_tH^2_x}\|u\|^{\alpha-\theta}_{B(\dot{H}^{s_c},I)} \|v\|_{B(\dot{H}^{s_c},I)}$,
\item [(ii)] $\left\||x|^{-b}|u|^\alpha v \right\|_{B'(L^2,I)}\lesssim \| u\|^{\theta}_{L^\infty_tH^2_x}\|u\|^{\alpha-\theta}_{B(\dot{H}^{s_c},I)} \| v\|_{B(L^2,I)},
$
\item [(iii)] $\left\|\nabla F(x,u)\right\|_{L^2_IL_x^{\frac{2N}{N+2}}}\lesssim 
\| u\|^{\theta}_{L^\infty_tH^2_x}\|u\|^{\alpha-\theta}_{S(\dot{H}^{s_c},I)} \|\Delta u\|_{B(L^2,I)},\; N >5 \text{ or } N = 5 \text{ and } \alpha < 7 - 2b,
$
\item [(iv)] $\left\|\nabla F(x,u)\right\|_{L^2_IL_x^{\frac{2N}{N+2}}}\lesssim 
\|u\|^{\theta}_{L^\infty_I H^2}
\|u\|^{\alpha_1}_{B(\dot{H}^{s_c})}
\|\langle \Delta \rangle u\|_{B(L^2)}^{1+\alpha-\alpha_1-\theta},\;N = 5 \;\; \text{ and }\; \alpha \geq 7-2b.
$

\end{itemize}
\end{lemma}

\begin{proof} For the estimates (i), (ii) and (iii), we refer the reader to  \cite[Section $4$]{GUZPAS} and \cite[Section $3$]{GuzPas1}. We prove here estimate (iv) for $7-2b \leq \alpha < 8-2b$ in dimension $N = 5$, therefore completing the whole intercritical range in this case. Indeed, note that
$$
|\nabla F(x, u)| \lesssim |x|^{-b} |\nabla (|u|^{\alpha}  u)| + |x|^{-b-1} |u|^{\alpha} |u|.
$$
We write $B$ for the unit ball centered at the origin, let $A\in\{B,B^C\}$, $0<\eta,\tilde{\theta}\ll 1$ and split (using H\"older)
\begin{equation}
    \| |x|^{-(b+1)} |u|^{\alpha}u\|_{L^{\frac{10}{7}}(A)} \leq \||x|^{-(b+1)}\|_{L^{r_1}(A)} 
    \|u\|^{\tilde{\theta}\eta}_{L^{r_2}} 
    \|u\|^{\alpha_1}_{L^{r_3}}
    \|u\|^{\alpha_2}_{L^{r_4}} \|u\|_{L^{r_5}},
\end{equation}
where

\begin{equation}
\begin{cases}
\frac{1}{r_1} = \frac{b+1+l \tilde{\theta} \eta}{5},\\
\frac{1}{r_2} = \frac{4-b}{5\alpha}-\frac{l}{5},  \\
\frac{1}{r_3} = \frac{4-b-\eta}{5\alpha},\\
\frac{1}{r_4} = \frac{\eta}{5\alpha},\\
\frac{1}{r_5} = \frac{\eta}{10},
\end{cases}, \quad
\begin{cases}
\alpha_1 = \frac{5-2b}{8-2b-4\eta}\alpha - \frac{3\alpha+\tilde{\theta}(8-2b)-2\tilde{\theta}\eta}{8-2b-4\eta} \eta,\\
\alpha_2 = \frac{3}{8-2b-4\eta}\alpha-\frac{\alpha-2\tilde{\theta}\eta}{8-2b-4\eta}\eta,
\end{cases}, \quad l=\begin{cases}
2-s_c,  &A = B,\\
-s_c, &A = B^C.
\end{cases}
\end{equation}
Observe that $0<\alpha_1<\alpha-\tilde{\theta}\eta$ in view of  $b<\frac{5}{2}$ and $\eta>0$ small. (Note that, unlike the previous works, we rely more on the Sobolev embeddings, by ``exchanging'' part of the power $\alpha$ in order to only work with admissible pairs. But as long as the remainder $\alpha_1$ is positive and we avoid the $L^\infty$ norm on time, the desired fixed-point argument can be closed.)  
Moreover, if $A=B$ then $\frac{5}{r_1}>b+1$ and if $A=B^C$, $\frac{5}{r_1}<b+1$, so we have that $|x|^{-b-1}\in L^{r_1}(A)$ in any case. Hence, applying the Sobolev inequality, it follows that\footnote{Note that $\frac{5\alpha}{\alpha s_c+\eta}<\frac{5}{s_c}$ and $\frac{10}{4+\eta}<\frac{5}{2}$ (conditions to apply the Sobolev inequality).}
\begin{equation}
\| |x|^{-(b+1)} |u|^{\alpha}u\|_{L^{\frac{10}{7}}(A)} \lesssim
\|u\|^{\tilde{\theta}\eta}_{H^2}
\|u\|^{\alpha_1}_{L^{\frac{5\alpha}{4-b-\eta}}}
\|D^{s_c} u\|^{\alpha_2}_{L^{\frac{5\alpha}{\alpha s_c+\eta}}} \|\Delta u\|_{L^{\frac{10}{4+\eta}}}.
\end{equation}
Since
$
\frac{1}{2}=\frac{\alpha_1\eta}{4\alpha}+\frac{\alpha_2(\alpha s_c+\eta)}{5\alpha}+\frac{4+\eta}{10}
$ and using again the H\"older inequality we deduce
\begin{equation}
\| |x|^{-(b+1)} |u|^{\alpha}u\|_{L^2_I L^{\frac{10}{7}}(A)} \lesssim
\|u\|^{\tilde{\theta}\eta}_{L^\infty_I H^2}
\|u\|^{\alpha_1}_{L^\frac{4\alpha}{\eta}_I L^{\frac{5\alpha}{4-b-\eta}}_x}
\|D^{s_c} u\|^{\alpha_2}_{L^{\frac{4\alpha}{4-b-\eta}}_I L^{\frac{5\alpha}{\alpha s_c+\eta}}_x} \|\Delta u\|_{L^\frac{8}{1-\eta}_I L^{\frac{10}{4+\eta}}_x}.
\end{equation}
It is easy to see that $\left(\frac{4\alpha}{\eta}, \frac{5\alpha}{4-b-\eta}\right)$ is a $B(\dot{H}^{s_c})$-admissible pair and $\left(\frac{4\alpha}{4-b-\eta},\frac{5\alpha}{\alpha s_c+\eta}\right)$, $\left({\frac{8}{1-\eta} ,{\frac{10}{4+\eta}}}\right)$ are $B$-admissible pairs\footnote{We observe that, $\alpha\geq 7-2b$ and $b<\frac{5}{2}$ implies that $\frac{4\alpha}{4-b-\eta}>2$, condition of admissible pair \eqref{L2Admissivel}.}, thus by interpolation
\begin{equation}
    \| |x|^{-(b+1)} |u|^{\alpha}u\|_{L^2_I L_x^{\frac{10}{7}}} \lesssim
    \|u\|^{\tilde{\theta}\eta}_{L^\infty_I H^2}
    \|u\|^{\alpha_1}_{B(\dot{H}^{s_c})}
    \|\langle \Delta \rangle u\|_{B(L^2)}^{1+\alpha_2}.
\end{equation}

The estimation of $ |x|^{-b} \nabla(|u|^{\alpha} u) \approx |x|^{-b} |u|^{\alpha} \nabla u$ is very similar, with the only changes being choosing $\frac{1}{r_1} =\frac{b+l \tilde{\theta} \eta}{5}$ and $\frac{1}{r_5} = \frac{\eta+2}{10}$, and so we conclude \begin{equation}
\| |x|^{-b} \nabla (|u|^{\alpha}u)\|_{L^2_I L_x^{\frac{10}{7}}} \lesssim
\|u\|^{\tilde{\theta}\eta}_{L^\infty_I H^2}
\|u\|^{\alpha_1}_{B(\dot{H}^{s_c})}
\|\langle \Delta \rangle u\|_{B(L^2)}^{1+\alpha_2}.
\end{equation}

Therefore combining the last two inequalities we obtain the desired result.
\end{proof} 

\begin{remark} We also have
\begin{equation}\label{useful estimate}
\||x|^{-b}|u|^{\alpha}u\|_{L^\infty_IL^{r}_x} \lesssim \|u\|_{L^\infty_I H^2_x}^{\alpha+1}
\end{equation}
for $\frac{2N}{N+8} < r < \frac{2N}{N+4}$.
\end{remark}
\begin{proof}
 Let $0<\eta \ll 1$. We write
 \begin{equation}
     \||x|^{-b}|u|^{\alpha}u\|_{L^{r}_x} \lesssim \| |x|^{-b} \|_{L_B^{\frac{N}{b+\eta}}} \|u\|^{\alpha+1}_{L_x^\frac{Nr(\alpha+1)}{N-r(b+\eta)}}
     +\| |x|^{-b} \|_{L_{B^c}^{\frac{N}{b-\eta}}} \|u\|^{\alpha+1}_{L_x^\frac{Nr(\alpha+1)}{N-r(b-\eta)}}.
 \end{equation}
 The conditions $\frac{2N}{N+8} < r < \frac{2N}{N+4}$ and $\frac{8-2b}{N}<\alpha < \frac{8-2b}{N-4}$ ensure 
 \begin{equation}
     2<\frac{Nr(\alpha+1)}{N-rb}<\frac{2N}{N-4},
 \end{equation}
 which, in turn, imply the embedding $H^2 \hookrightarrow L^{\frac{Nr(\alpha+1)}{N-r(b \pm \eta)}}$.
\end{proof}

%It is worth mentioning that in the previous lemma was first obtained in \cite{GUZPAS}, assuming extra conditions (on item (iii)) on $\alpha$ for dimension $N=5,6,7$. Recently, these assumptions were removed in \cite{GuzPas1}. 

We end this section with an important lemma, followed by the proof of Theorem \ref{GWPH2}. The full-range nonlinear estimates also play a key role in obtaining these results. %\ref{Scattering}.

\begin{lemma}[{Space-time bounds imply scattering}]\label{bound-scattering} Let $N$, $\alpha$ and $b$ be as in Lemma \ref{L:NL}. Let $u$ be a global solution to \eqref{IBNLS} satisfying $\|u\|_{L^\infty_t H^2_x}\leq E$. If 
\begin{equation*}
\|u\|_{B(\dot{H}^{s_c},[T,+\infty))} <+\infty,
\end{equation*}
for some $T>0$, then
 $u$ scatters forward in time in $H^2$.%, i.e., here exists $u_+ \in H^2$ such that
%\begin{equation}
%	\lim_{t\to+\infty} \|u(t)-e^{it\Delta}u_+\|_{H^2_x} = 0.
%\end{equation}
\begin{proof} 

For $\eta>0$, let $[T,+\infty) = \displaystyle\bigcup_{j=1}^N I_j$, in which the intervals $I_j$ are chosen such that $\|u\|_{B(\dot{H}^{s_c},I_j)}<\eta$ for all $j$. If $[a,a+t]\subset I_j$, by Strichartz, and Lemma \ref{L:NL}, there exists $0<\alpha_1<\alpha$ such that

\begin{equation}
    \|\langle\Delta\rangle u\|_{S(L^2,[a,a+t])} \lesssim \|u(a)\|_{H^2} + \|u\|_{B(\dot{H}^{s_c},I_j)}^{\alpha_1} \|\langle \Delta \rangle u\|_{S(L^2,[a,a+t])}^{\alpha-\alpha_1+1} \leq  E + \eta^{\alpha_1} \|\langle \Delta \rangle u\|_{S(L^2,[a,a+t])}^{\alpha-\alpha_1+1}. 
\end{equation}

This implies, by a continuity argument, $\|\langle\Delta\rangle u\|_{S(L^2,I_j)} \lesssim E$, if $\eta$ is chosen such that $\eta \ll E^{-\frac{\alpha-\alpha_1}{\alpha_1}}$. Summing from $j = 1$ to $N$, we conclude

\begin{equation}
     \|\langle\Delta\rangle u\|_{S(B^2,[T,+\infty))} < +\infty.
\end{equation}

%Applying Lemma \ref{small_data} we have
%$$
%\|u\|_{B(\dot{H}^{s_c},[T,+\infty))} \leq 2 	\|e^{it\Delta^2}u(T)\|_{B(\dot{H}^{s_c},[T,+\infty))}\leq 2\delta_{sd}
%$$
%and 
%$$
%\|\langle \Delta \rangle u\|_{B(L^2, [T,+\infty))} \leq 2c \|u(T)\|_{H^2}\leq 2cE.
%$$
Now, define
$$
\phi^+=e^{-iT\Delta^2}u(T)+i\int\limits_{T}^{+\infty}e^{-is\Delta^2}\left(|x|^{-b}|u|^\alpha u\right)(s)ds.
$$
We see that $\phi^+ \in H^2$, since combining Strichartz estimates and Lemma \ref{L:NL} gives

\begin{equation}
\|\phi^+\|_{H^2} \lesssim \|u\|_{L^\infty_{[T,+\infty)}H^2_x} + \| u \|^{\alpha_1}_{B(\dot{H}^{s_c};[T,\infty))}\|\langle \Delta \rangle u\|_{B(L^2,[T,\infty))}^{\alpha-\alpha_1-1} < +\infty.
\end{equation}

A simple inspection shows
$$
 u(t)-e^{it\Delta^2}\phi^+=i\int\limits_{t}^{+\infty}e^{i(t-s)\Delta^2}|x|^{-b}(|u|^\alpha u)(s)ds,
$$
thus again by Strichartz and Lemma \ref{L:NL}, it follows that 
$$
 \|u(t)-e^{it\Delta^2}\phi^+\|_{H^2}\lesssim \| u \|^{\alpha_1}_{B(\dot{H}^{s_c};[t,\infty))}
 \|\langle \Delta \rangle u\|_{B(L^2,[t,\infty))}^{\alpha-\alpha_1+1}.
$$
Since $ \| u \|_{B(\dot{H}^{s_c};[T,\infty))}<+\infty$, we conclude that 
\begin{equation*}%\label{scat1}
\|u(t)-e^{it\Delta^2}\phi^+\|_{H^2}\rightarrow 0, \,\,\textnormal{as}\,\,t\rightarrow +\infty.
\end{equation*}

\end{proof}
\end{lemma}

\begin{proof}[\bf Proof of Theorem \ref{GWPH2}] 
We only prove the case $N=5$, $7-2b \leq \alpha < 8-2b$, since the remaining cases can have the same treatment as in [7,
Theorem 1.6]. Given $E >0$ and $v_0 \in H^2$ such that $\|v_0\|_{H^2} \leq E$, define the metric space 
\begin{equation}
F = \left\{v \,|\, 
%\|v\|_{L^\infty_tH^2_x} \leq E,\,
\|v\|_{B(\dot{H}^{s_c})}  \leq 2 \|e^{ it \Delta^2} v_0\|_{B(\dot{H}^{s_c})}, \, \|v\|_{S(L^2)}+\|\Delta v\|_{S(L^2)} \leq 2c \|v_0\|_{H^2}  \right\},
\end{equation}
where $c$ is the constant given by Strichartz estimates. Equip $F$ with the distance
\begin{equation}
    d(u,v) = \|u-v\|_{S(\dot{H}^{s_c})}.
\end{equation}

By completeness of $L^p$ spaces, reflexiveness and uniqueness of weak and strong limits, $F$ is a complete space. Now, define the map
\begin{equation}
G(v) = e^{it\Delta^2}v_0 + i \int_0^t e^{i(t-s)\Delta^2}|x|^{-b}|v|^{\alpha} v(s) \, ds.
\end{equation}
We want to show that $G$ maps $F$ in $F$ and it is a contraction. Indeed, combining the Strichartz estimates together with Lemma \ref{L:NL}, one has

\begin{align}
\|G(v)\|_{S(\dot{H}^{s_c})} &\leq \|e^{ it \Delta^2} v_0\|_{B(\dot{H}^{s_c})} + cE^\theta \|v\|^{\alpha-\theta+1}_{S(\dot{H}^{s_c})}\\ &\leq [1+cE^\theta(2\delta)^{\alpha-\theta}  ]\|e^{ it \Delta^2} v_0\|_{B(\dot{H}^{s_c})},\\
\|\langle \Delta \rangle G(v)\|_{S(L^2)} &\leq c\|v_0\|_{H^2} + cE^\theta \|v\|^{\alpha_1}_{S(\dot{H}^{s_c})} 
\|\langle \Delta \rangle v\|_{S(L^2)}^{1+\alpha-\alpha_1-\theta}
%\leq c[1+ c(2c)^{\alpha-\alpha_1-\theta}E^{\alpha-\alpha_1}(2\delta)^{\alpha_1}]  \|v_0\|_{H^2},
\\
&\leq c[1+ c(2c)^{\alpha-\alpha_1-\theta}E^{\alpha-\alpha_1}(2\delta)^{\alpha_1}]  \|v_0\|_{H^2}\end{align}
and
\begin{align}
\|G(u) - G(v)\|_{S(\dot{H}^{s_c})} &\leq c E^{\theta} \left(\|u\|^{\alpha-\theta}_{S(\dot{H}^{s_c})} + \|v\|^{\alpha-\theta}_{S(\dot{H}^{s_c})}\right) \|u-v\|_{S(\dot{H}^{s_c})} \\
&\leq 2c E^{\theta} (2\delta)^{\alpha-\theta} \|u-v\|_{S(\dot{H}^{s_c})}.  
\end{align}
Therefore, by choosing a small $\delta$ (depending only on $E$), the theorem is proved.
\end{proof}

\ 

\section{Scattering criterion}

In this section is devoted to show Proposition\ref{scattering_criterion}. We start with the following lemma that will be used in the proof. 
\begin{lemma}\label{linevo}
Let  $N \geq 5$, $0 < b < \min\{\frac{N}{2},4\}$, $\frac{8-2b}{N} < \alpha < \frac{8-2b}{N-4}$ and $u$ be a (possibly non-radial) $H^2$-solution to \eqref{IBNLS} satisfying \eqref{E}. If $u$ satisfies \eqref{scacri} for some $0 < \epsilon < 1$, then there exist $\gamma, T > 0$ such that 
\begin{equation}\label{norm-small}
\left\|e^{i(\cdot-T)\Delta}u(T)\right\|_{B\left(\dot{H}^{s_c}, [T, +\infty)\right)}  \lesssim\epsilon^\gamma.
\end{equation}
\end{lemma}

\begin{proof} 
Fix the parameters $\mu, \gamma >0$ (to be chosen later). Applying the Strichartz estimate \eqref{SE1}, there exists $T_{0} > \epsilon^{-\mu}$ such that
\begin{equation}\label{T0}
\left\|e^{it\Delta}u_0\right\|_{B\left(\dot{H}^{s_c}, [T_0,+\infty) \right)} \leq \epsilon^{\gamma}.
\end{equation}

For $T\geq T_0$ to be chosen later, define  $I_1 :=\left[T-\epsilon^{-\mu}, T\right]$, $I_2 := [0, T-\epsilon^{-\mu}]$  and let $\eta$ denote a smooth, spherically symmetric function which equals $1$ on $B(0, 1/2)$ and $0$ outside $B(0,1)$. For any $R > 0$ use $\eta_R$ to denote the rescaling $\eta_R(x) := \eta(x/R)$. 

Duhamel's formula implies that
%\begin{equation}
%u(T) = e^{iT\Delta^2}u_0 - i\int_0^{T}e^{i(T-s)\Delta^2}|x|^{-b}|u|^{\alpha}u(s) \, ds,
%\end{equation}
%we obtain
\begin{equation}
e^{i(t-T)\Delta^2}u(T)  = e^{it\Delta^2}u_0 - iF_1 - iF_2,
\end{equation}
where, for $i = 1,2,$
$$
F_i = \int_{I_i} e^{i(t-s)\Delta}|x|^{-b}|u|^{\alpha}u(s) \,ds.
$$
We refer to $F_1$ as the ``recent past", and to $F_2$ as the ``distant past". By \eqref{T0}, it remains to estimate $F_1$ and $F_2$.

%\textbf{Step 1. Estimate on recent past.}

We start with $F_1$. By hypothesis \eqref{scacri}, we can fix $T\geq T_0$ such that
\begin{equation}\label{mass}
\int \eta_R(x)\left|u(T,x)\right|^2dx\lesssim \epsilon^2.
\end{equation}
Given the relation  (obtained by multiplying \eqref{IBNLS} by $\eta_R\bar{u}$ , taking the imaginary part and integrating by parts)%, see Tao \cite[Section 4]{Tao} for the NLS version of this identity)

$$
\partial_t\int \eta_R|u|^2\, dx = 2\Im\left(\int\Delta\eta_R \Delta u \bar{u}+\int \nabla \eta_R \cdot \nabla \bar{u} \Delta u\right),
$$
we have, from \eqref{E}, for all times,
$$
\left| \partial_t \int \eta_R(x)|u(t,x)|^2dx\right| \lesssim \frac{1}{R},
$$

so that,  by \eqref{mass}, for $t \in I_1$,
\begin{equation}
    \int \eta_R(x)\left|u(t,x)\right|^2dx\lesssim \epsilon^2+\frac{\epsilon^{-\mu}}{R}.
\end{equation}

If $R > \epsilon^{-(\mu+2)}$, then we have $\left\| \eta_Ru\right\|_{L^\infty_{I_1}L^2_x} \lesssim 
\epsilon
$.

%Let  $(q, r)\in \mathcal{A}_{s_c}$ . Recalling that $2 + \delta \leq r \leq p^*-\delta$ (see Remark \ref{delta}), using interpolation and Sobolev inequalities and the decay of the $L^\infty$ norm of radial functions outside the ball \eqref{Strauss}, we get

Now, we use the pair $(\tilde{a},r) \in \mathcal{B}_{-s_c}$ used in \cite[Lemma $4.2$]{GUZPAS} given by
\begin{align}
\widetilde{a} = \frac{8\alpha(\alpha+2-\theta)}{\alpha[N\alpha+2b]-\theta[N\alpha-8+2b]},\,\,
r = \frac{N(\alpha (\alpha+2-\theta)}{\alpha(N-b)-\theta(4-b)}.
\end{align}
By using the H\"older and Sobolev inequalities, for $t \in I_1$, we deduce that\footnote{See the proof Lemma $4.2$ in \cite{GUZPAS}, for more details.}

\begin{equation}\label{recent_past_ball1}
||\, \eta_R |x|^{-b}|u|^{\alpha}u(t) ||_{L_x^{r'}} \lesssim  \|u(t)\|^\theta_{H^2_x} \|u(t)\|^{\alpha-\theta}_{L_x^{r}}\|\eta_R u(t)\|_{L_x^{r}} \lesssim \|\eta_R u(t)\|_{L_x^{r}}.
\end{equation}
Letting $\hat{\theta}$ be the solution of $\frac{1}{r} = \frac{\hat{\theta}}{2}+\frac{1-\hat{\theta}}{p^*}$, we have,
\begin{equation}\label{recent_past_ball2}
\|\eta_R u(t)\|_{L_x^{r}} \leq \|u(t)\|^{1-\hat{\theta}}_{L^{p^*}_x}\|\eta_R u(t)\|^{\hat{\theta}}_{L^2_x} \lesssim \epsilon^{\hat{\theta}},
\end{equation}
uniformly on time in $I_1$. We now exploit the decay of the nonlinearity, instead of assuming radiality\footnote{It is one of the crucial estimates which allow us to drop the radiality assumption.}, to estimate, by H\"older and Sobolev, for $R(\epsilon)>0$ large enough and $t \in I_1$,
\begin{align}\label{recent_past_out}
||\, (1-\eta_R) |x|^{-b}|u|^{\alpha}u(t) ||_{L_x^{r'}} &\leq ||\,|x|^{-b}|u|^{\alpha}u(t) ||_{L_{\{|x|>R/2\}}^{r'}} \nonumber\\&\leq \|\, |x|^{-b}\|_{L_{\{|x|>R/2\}}^{r_1}} \|u(t)\|^\theta_{L^{\theta r_2}_x} \|u(t)\|^{\alpha+1-\theta}_{L_x^{r}}\nonumber\\
& \lesssim \frac{1}{R^{br_1-N}} \|u(t)\|^{\alpha+1}_{H^2_x} \lesssim \epsilon^{\hat{\theta}},
\end{align}
where $r_1$ and $r_2$ are such that $br_1 > N$, $\theta r_2 \in (2,N\alpha/(4-b))$ and
\begin{equation}
\frac{1}{r'} = \frac{1}{r_1} + \frac{1}{r_2} + \frac{\alpha+1-\theta}{r}.
\end{equation}

%\begin{align}\label{I1}
%\|u\|_{L^\infty_{I_1}L^{r}_x}
%&\lesssim  \left\| \eta_Ru\right\|^{\frac{2\left(p^*-r\right)}{r\left(p^*-2\right)}}_{L^\infty_{I_1}L^2_x} \left\| \eta_R u\right\|_{L^\infty_{I_1}L^{p^*}_x}^{1-\frac{2\left(p^*-r\right)}{r\left(p^*-2\right)}} + \left\|(1-\eta_R)u\right\|^{\frac{r-2}{r}}_{L^\infty_{I_1}L_x^\infty}\left\|(1-\eta_R)u\right\|^\frac{2}{r}_{L^\infty_{I_1}L^2_x}
%\nonumber\\
%
%&\lesssim  \epsilon^{\frac{2\left(p^*-r\right)}{r\left(p^*-2\right)}}
%\left\| u\right\|_{L^\infty_{I_1}L^{p^*}_x}^{1-\frac{2\left(p^*-r\right)}{r\left(p^*-2\right)}} + R^{-\frac{N-1}{2}\left( \frac{r-2}{r}\right)}\|u\|_{L_t^{\infty}H_x^1}^{\left( \frac{r-2}{r}\right)}\left\|u_{0}\right\|^\frac{2}{r}_{L^2_x}\\
%
%&\lesssim \epsilon^{\frac{2\delta}{(p^*-\delta)(p^*-2)}}+ R^{-\frac{N-1}{2} \frac{\delta}{p^*-\delta}}
%\lesssim \epsilon^{\frac{2\delta}{(p^*-\delta)(p^*-2)}},
%\end{align}

%if $R$ is large enough.
%Note that, in the penultimate step, we used the $H^1\hookrightarrow L^{p^*}$ embedding. 
Combining the Strichartz estimate \eqref{SE4}, together with estimates \eqref{recent_past_ball1}, \eqref{recent_past_ball2} and \eqref{recent_past_out}, one has

\begin{align}
\left\| \int_{I_1} e^{i(t-s)\Delta}|x|^{-b}|u|^{\alpha}u(s) \,ds\right\|_{B(\dot{H}^{s_c},[T,+\infty))} &\lesssim ||\, |x|^{-b}|u|^{\alpha}u ||_{B'(\dot{H}^{-s_c},I_1)}\\
&\hspace{-5cm}\leq ||\, \eta_R |x|^{-b}|u|^{\alpha}u ||_{L^{\widetilde{a}'}_{I_1}L_x^{r'}} + ||\, (1-\eta_R) |x|^{-b}|u|^\alpha u ||_{L^{\widetilde{a}'}_{I_1}L_x^{r'}}\\
&\hspace{-5cm}\lesssim|I_1|^{1/\widetilde{a}'}\epsilon^{\hat{\theta}} = \epsilon^{\hat{\theta}-\mu/\widetilde{a}'}= \epsilon^{\hat{\theta}/2},
\end{align}
where we choose $\mu := \widetilde{a}'\hat{\theta}/{2}$.
 
%\begin{align}
%\left\| \int_{I_1} e^{i(t-s)\Delta}|x|^{-b}|u|^{p-1}u(s) \,ds\right\|_{S(\dot{H}^{s_c},[T,+\infty))} &\leq ||\, |x|^{-b}|u|^{p-1}u ||_{S'(\dot{H}^{-s_c},I_1)}\\
%&\hspace{-5cm}\leq  \|u\|^\theta_{L^\infty_tH_x^1} \|u\|^{p-\theta}_{S(\dot{H}^{s_c},I_1)}= \left\|u\right\|^\theta_{L^\infty_tH_x^1} \,\,  \sup_{(q,r)\in \mathcal{A}_{s_c}}\|u\|^{p-\theta}_{L^{q}_{I_1}L^{r}_x}\\
%
%&\hspace{-5cm}\leq \left\|u\right\|^\theta_{L^\infty_tH_x^1} \,\,  \sup_{2+\delta \leq r \leq p^*-\delta}\|u\|^{p-\theta}_{L^{\infty}_{I_1}L^{r}_x}\epsilon^{-\alpha\left(\frac{p-\theta}{q}\right)}\\
%
%&\hspace{-5cm}\leq   \left\|u\right\|^\theta_{L^\infty_tH_x^1} 
%\epsilon^{\frac{2\delta}{(p^*-\delta)(p^*-2)}
%({p-\theta})}\epsilon^{-\alpha\left(\frac{p-\theta}{2+\delta}\right)}\lesssim \epsilon^{\frac{\delta(p-\theta)}{(p^*-\delta)(p^*-2)}} 
%\lesssim \epsilon^\gamma
%,
%\end{align}
%where we used $\alpha = {\frac{\delta(2+\delta)}{(p^*-\delta)(p^*-2)}}$.
%and $\gamma \leq \frac{\delta(p-\theta)}{(p^*-\delta)(p^*-2)} $.
%where we used the definition of $\alpha >0$ and the fact that $q \geq 2+\delta$.

%\textbf{Step 2. Estimate on distant past.}

We now estimate $F_2$. %The estimate for the distant past is the same as in \cite{Campos_New_2019}, as radiality does not play a role in this part of the estimate. We provide the argument here for completeness. 
Let  $(a, r) \in \mathcal{B}_{s_c}$ and define
\begin{equation}
\frac{1}{c} = \left(\frac{1}{2-s_c}\right)\left[\frac{2}{a}-\delta s_c\right]
\end{equation}
and
\begin{equation}
\frac{1}{d} = \left(\frac{1}{2-s_c}\right)\left[\frac{2}{r}-s_c\left(\frac{N-4-8\delta}{2N}\right)\right],
\end{equation}
where $\delta>0$ is small. It is easy to see that the pair $(c,d)$ is $B$-admissible\footnote{Since $a>\frac{4}{2-s_c}$ we have that $c>2$, which implies $d<\frac{2N}{N-4}$, that is, the pair $(c,d)$ satisfies \eqref{L2Admissivel}.}. By interpolation,

$$\left\|F_2\right\|_{L_{[T,+\infty)}^{a}L_x^{r}}  
\leq 
\left\|F_2\right\|^{\frac{2-s_c}{2}}_{L_{[T,+\infty)}^{c}L_x^{d}}
\left\|F_2\right\|^{\frac{s_c}{2}}_{L_{[T,+\infty)}^{\frac{1}{\delta}}L_x^{\frac{2N}{N-4-8\delta}}}.
$$ 
we can rewrite $F_2$ by (applying Duhamel's principle)
$$
F_2 = e^{it\Delta}\left[e^{i(-T+\epsilon^{-\mu})\Delta}u(T-\epsilon^{-\mu})-u(0)\right].
$$
The Strichartz estimate \eqref{SE1}, with $s=0$, leads to
\begin{align}
\left\| F_2\right\|_{L_{[T,+\infty)}^{a}L_x^{r}}
&\leq\left\| e^{it\Delta}\left[e^{i(-T+\epsilon^{-\mu})\Delta}u(T-\epsilon^{-\mu})-u(0)\right]\right\|^{\frac{2-s_c}{2}}_{L_{[T,+\infty)}^{c}L_x^{d}}
\left\|F_2\right\|^{\frac{s_c}{2}}_{L_{[T,+\infty)}^{\frac{1}{\delta}}L_x^{\frac{2N}{N-4-8\delta}}}\\
&\lesssim\left(\left\|u\right\|_{L^\infty_t L^2_x}\right)^{\frac{2-s_c}{2}} \left\|F_2\right\|^{\frac{s_c}{2}}_{L_{[T,+\infty)}^{\frac{1}{\delta}}L_x^{\frac{2N}{N-4-8\delta}}}
\lesssim \epsilon^{\frac{\mu \delta s_c}{2}}.
\end{align}
The estimate \eqref{useful estimate} and the free Schr\"odinger operator decay
$$\|e^{it\Delta^2} \cdot \|_{L^r} \lesssim \frac{1}{ t^{\frac{N}{4}(1-\frac{2}{r})}}\| \cdot\|_{L^{r'}},\;\; \forall r \geq 2,
$$
yield
\begin{align}
\left\|F_2\right\|_{L_{[T,+\infty)}^{\frac{1}{\delta}}L_x^{\frac{2N}{N-4-8\delta}}}
&\lesssim \left\|\int_{I_2} |\cdot-s|^{-(1+2\delta)}\left\| |x|^{-b} |u|^{\alpha}u(s) \right\|_{L^{\frac{2N}{N+4+8\delta}}_x}\, ds\right\|_{L_{[T,+\infty)}^{\frac{1}{\delta}}}\\
&\lesssim \|u\|_{L_t^\infty H_x^2}^{\alpha+1}\left\|\left(\cdot -T+\epsilon^{-\mu}\right)^{-2\delta}\right\|_{L_{[T,+\infty)}^{\frac{1}{\delta}}}\\
&\lesssim \epsilon^{\mu\delta}.
\end{align}

Finally, defining $\gamma := \min\{ \frac{\hat{\theta}}{2}, \frac{\mu \delta s_c}{2}\}$ and recalling that

$$
e^{i(t-T)\Delta}u(T) = e^{it\Delta}u_0 + iF_1 + iF_2,
$$
we obtain \eqref{norm-small}. 
%conclude that
%$$
%\left\|e^{i(\cdot-T)\Delta} u(T)\right\|_{S\left(\dot{H}^{s_c}, [T,+\infty) \right)}  \lesssim\epsilon^\gamma.
%$$
%Hence, Lemma \ref{linevo} is proved.
\end{proof}
\begin{proof}[\bf Proof of Proposition \ref{scattering_criterion}]
%We will now prove that $\left\| u\right\|_{S\left(\dot{H}^{s_c}, [T, +\infty)\right)}  \lesssim\epsilon^\gamma $.

Choose $\epsilon$ is small enough so that, by Lemma \ref{linevo}, $$\left\|e^{i(\cdot)\Delta}u(T)\right\|_{S\left(\dot{H}^{s_c}, [0, +\infty)\right)}  = \left\|e^{i(\cdot-T)\Delta} u(T)\right\|_{S\left(\dot{H}^{s_c}, [T, +\infty)\right)} \leq c\epsilon^\gamma\leq \delta_{sd},
$$
which implies that the norm $
\|u\|_{B(\dot{H}^{s_c},[0,+\infty))}$ is bounded, by Theorem \ref{GWP}. Thus, using Lemma \ref{bound-scattering} we conclude that $u$ scatters forward in time in $H^2$.
\end{proof}

%%%%%%%%%%%%%%%%%%%%%%%%%%%%%%%%%%%%%%%%%%%%%%%%%%%%%%%%%%%%%%%%%%%%%%%%%%%%%%%%%%%%%%%%%%%%%%%%%%%%%%%%%%%%%%%%%%%%%%

\

\section{Scattering: Proof of Theorem \ref{Scattering}}

To prove scattering, we first obtain some local coercivity results. Then we prove a Virial-Morawetz-type estimate to gain control over a suitable norm on large balls. The proof is concluded by using the scattering criterion. We remark here that we do not make any radiality assumption, and instead take advantage of the decay of the nonlinearity.

\subsection{Coercivity}
We recall here the so-called coercivity (also known as \textit{energy-trapping}) results for the IBNLS, which were proved in \cite{SaanouniRadial}.

\begin{lemma}\label{lem_coerc_1}
Let $N$, $\alpha$ and $b$ as in Theorem \ref{Scattering}, and $f \in H^2(\mathbb{R}^N)$. Assume that, for some $\delta_0 > 0$,
\begin{equation}
M(f)^{\frac{2-s_c}{s_c}} E(f) \leq (1-\delta_0) M(Q)^{\frac{2-s_c}{s_c}}E(Q), 
\end{equation}
and
\begin{equation}
    \|f \|^{\frac{2-s_c}{s_c}}_{L^2}\|\Delta f \|_{L^2}  \leq \|Q\|^{\frac{2-s_c}{s_c}}_{L^2} \|\Delta Q\|_{L^2}.
    \end{equation}
    Then there exists $\delta = \delta(\delta_0, N, p, Q)$ such that
    \begin{equation}
        \|f \|^{\frac{2-s_c}{s_c}}_{L^2}\|\Delta f \|_{L^2}  \leq (1-\delta) \|Q\|^{\frac{2-s_c}{s_c}}_{L^2} \|\Delta Q\|_{L^2}.
    \end{equation}

\end{lemma}

\begin{lemma}\label{lem_coerc_2}
Under the conditions of the previous lemma, one also has, for some $\eta > 0$,
\begin{equation}
    \int  \left[|\Delta f|^2
    -  \frac{N\alpha+2b}{4(\alpha+2)}|x|^{-b}|f|^{\alpha+2} \right]\, dx \geq \eta  \int |x|^{-b}|f|^{\alpha+2}\, dx.
\end{equation}
\end{lemma}

%The following result is an immediate consequence of Lemma \ref{lem_coerc_1} and energy conservation.
%\begin{lemma} Let $N$, $\alpha$ and $b$ as in Theorem \ref{Scattering}, and $u_0 \in H^2(\mathbb{R}^N)$. Denote by $I$ the maximal time of existence of the corresponding solution $u$ to \eqref{IBNLS}. If
%\begin{equation}\label{Cond1_prev}
%M(u_0)^{\frac{2-s_c}{s_c}} E(u_0) < M(Q)^{\frac{2-s_c}{s_c}}E(Q)
%\end{equation}
%and
%\begin{equation}\label{Cond2_prev}
%\|u_0 \|^{\frac{2-s_c}{s_c}}_{L^2}\|\Delta u_0 \|_{L^2}  < \|Q\|^{\frac{2-s_c}{s_c}}_{L^2} \|\Delta Q\|_{L^2}.
%    \end{equation}

% Then, 
% \begin{equation}
%\sup_{t \in I}\|u_0 \|^{\frac{2-s_c}{s_c}}_{L^2}\|\Delta u(t) \|_{L^2}  < \|Q\|^{\frac{2-s_c}{s_c}}_{L^2} \|\Delta Q\|_{L^2}.
% \end{equation}
%In particular, $E := \sup_{t \in I} \|u(t)\|_{H^2} < +\infty$, which implies $I = (-\infty, +\infty)$.
%\end{lemma}

\subsection{Virial-Morawetz estimate}

\begin{proposition}\label{virial}

%[Virial-Morawetz estimate]\label{virial}
For $N$, $\alpha$ and $b$ as in Theorem \ref{Scattering} , let $u$ be a $H^2(\mathbb{R}^N)$-solution to \eqref{IBNLS} satisfying \eqref{cond1} and \eqref{cond2}. Then, for any $T>0$,
\begin{equation}\label{virial-morawetz-decay}
\frac{1}{T}\int_0^T\int|x|^{-b}|u(x,t)|^{\alpha+2}\,dx\, dt \lesssim \frac{1}{T^\frac{{\min\{2,b\}}}{1+{\min\{2,b\}}}}.     
\end{equation}
\end{proposition}
\begin{proof}

Let $R \gg 1$   to be determined below. We take $a$ to be a smooth radial function satisfying
$$
a(x) = \begin{cases}|x|^2, & |x|\leq \frac{1}{2}, \\
|x|,& |x| > 1. \\
\end{cases}
$$
In the intermediate region $\frac{1}{2} < |x| \leq 1$, we impose that
$$
\partial_ra \geq 0, \,\,\, \partial_r^2a \geq 0.
$$
Here, $\partial_r$ denotes the radial derivative, i.e., $\partial_r a = \nabla a \cdot \frac{x}{|x|}$. Note that for $|x| \leq \frac{1}{2}$, we have
$$
a_{ij} = 2 \delta_{ij}, \,\,\, \Delta a = 2N, \,\,\,\text{and } \partial^{\beta}a = 0 \text{ for } |\beta|\geq3.
$$

Finally, define

\begin{equation}
    a_R(x) = R^2 \,a(x/R).
\end{equation}

Consider now the Virial/Morawez quantity
\begin{equation}
    Z(t) := \Im\int  \nabla a_R \cdot \nabla u \,\bar{u}\, dx. 
\end{equation}
By Cauchy-Schwarz, one has $\displaystyle\sup_t |Z(t)| \lesssim R$. We now make use of the virial identity (see \cite{Guo} and \cite{Pausader07}):

\begin{align}\label{main_vir}
    Z'(t) &=  
    - 4\sum_{i,j,k}\int\partial_{jk}a_R \partial_{ik} \bar{u} \partial_{ij} u \, dx
    + \int \left(\frac{\alpha}{\alpha+2}\Delta a_R+\frac{2b }{\alpha+2}\frac{x \cdot \nabla a_R}{|x|^2} \right)|x|^{-b}|u|^{\alpha+2} \, dx\\\label{err_vir}& \quad 
    + 2\sum_{j,k} \int  \partial_{jk}\Delta a_R \partial_j \bar{u} \partial_k u \, dx
    - \frac{1}{2}\int\Delta^3 a_R |u|^2 \, dx
    + \int\Delta^2 a_R |\nabla u|^2 \, dx.
\end{align}

By Cauchy-Schwarz and the definition of $a_R$ (together with the classical chain rule), one has
\begin{equation}\label{err_vir2}
    |\eqref{err_vir}| \lesssim \frac{1}{R^2}.
\end{equation}

For the main term, we compute
\begin{align}
    \eqref{main_vir} 
    &= -4\int_{|x|\leq \frac{R}{2}}  \left[|\Delta u|^2  
    -  \frac{N\alpha+2b}{4(\alpha+2)}|x|^{-b}|u|^{\alpha+2} \right]\, dx
    \\
    &\quad -2\int_{ |x| > \frac{R}{2}}\partial_r^2a|\nabla \partial_r u|^2
    -2\sum_i\int_{ |x| > \frac{R}{2}}\frac{\partial_r a}{|x|}|\slashed{\nabla} \partial_i u|^2+O(\int_{|x| > \frac{R}{2}} |x|^{-b}|u|^{\alpha+2})\\
    &\leq -4\int_{|x|\leq \frac{R}{2}}  \left[|\Delta u|^2  
    -  \frac{N\alpha+2b}{4(\alpha+2)}|x|^{-b}|u|^{\alpha+2} \right]\, dx+O(\frac{1}{R^b}),
\end{align}
where the angular derivative is defined as $\slashed{\nabla} u = \nabla u - \frac{x\cdot \nabla u}{|x|^2}x$. The terms $\slashed{\nabla} \partial_i u$ are not necessarily zero, since we are not assuming radiality, but the corresponding integrals can be discarded for having a non-negative sign. Therefore,
\begin{equation}
    -Z'(t) \geq\int_{|x|\leq \frac{R}{2}}  \left[|\Delta u|^2  
    -  \frac{N\alpha+2b}{4(\alpha+2)}|x|^{-b}|u|^{\alpha+2} \right]\, dx+ O(\frac{1}{R^{\min\{2,b\}}}),
\end{equation}

which, by integration on time, gives
\begin{equation}\label{local_morawetz}
    \int_0^T\int_{|x|\leq \frac{R}{2}}  \left[|\Delta u|^2  
    -  \frac{N\alpha+2b}{4(\alpha+2)}|x|^{-b}|u|^{\alpha+2} \right]\, dx dt \lesssim R +\frac{T}{R^{\min\{2,b\}}},
\end{equation}

We now show that there exists $\eta > 0$ such that

\begin{equation}\label{local_coerc_virial}
 \int_{|x|\leq \frac{R}{2}}  \left[|\Delta u|^2  
    -  \frac{N\alpha+2b}{4(\alpha+2)}|x|^{-b}|u|^{\alpha+2} \right]\, dx \geq \eta \int_{|x|\leq \frac{R}{2}} |x|^{-b}|u|^{\alpha+2}\, dx + O(\frac{1}{R^2}).
\end{equation}
Indeed, if $\phi^A$ is a smooth cutoff to the set $\{|x|\leq \frac{1}{2}\}$ that vanishes outside $\{|x|\leq \frac{1}{2}+\frac{1}{A}\}$, define $ \chi_R^A(x) := \phi^A(\frac{|x|}{R})$. We then have
\begin{equation}\label{commut}
    \Delta(\chi_R^A u) = \chi_R^A \Delta u + 2 \nabla \chi_R^A \cdot \nabla u + u \Delta \chi_R^A
    =\chi_R^A \Delta u+ O(\frac{1}{R}),
\end{equation}
so that, if $\delta_0 > 0$ is such that $M(u_0)^{\frac{2-s_c}{s_c}}E(u_0) \leq (1-\delta_0) M(Q)^{\frac{2-s_c}{s_c}}E(Q)$, then
\begin{equation}
    M(\chi_R^A u)^{\frac{2-s_c}{s_c}}E(\chi_R^A u) \leq M(u_0)^{\frac{2-s_c}{s_c}}E(u_0) + \frac{C}{R } \leq (1-\frac{\delta_0}{2})M(Q)^{\frac{2-s_c}{s_c}}E(Q)
\end{equation}

\begin{equation}
    \|\chi_R^A u\|_{L^2}^{\frac{2-s_c}{s_c}}\|\Delta (\chi_R^A u)\|_{L^2} \leq \| u\|_{L^2}^{\frac{2-s_c}{s_c}}\|\Delta u\|_{L^2} + \frac{C}{R} \leq \| Q\|_{L^2}^{\frac{2-s_c}{s_c}}\|\Delta Q\|_{L^2}, 
\end{equation}
if $R>0$ is large enough (uniformly on time).
Therefore, by Lemma \ref{lem_coerc_2},
\begin{equation}
    \int \left[|\Delta(\chi_R^A u)|^2  
    - \frac{N\alpha+2b}{4(\alpha+2)}|\chi_R^A u|^{\alpha+2} \right]\, dx \geq \eta \int|\chi_R^A u|^{\alpha+2} \, dx.
\end{equation}

Now, 
%it only remains to estimate
%\begin{equation}\label{remainder_morawetz}
%    \int\left||\chi_R^A|^{2} - |\chi_R^A|^{\alpha+2}\right| |x|^{-b}|u|^{\alpha+2}\, dx \lesssim \int_{|x|>\frac{R}{2}} |x|^{-b}|u|^{\alpha+2} 
    %\lesssim \frac{1}{R^{\frac{(N-1)\alpha}{2}}} 
%    \lesssim \frac{1}{R^b},
%\end{equation}
%and, 
by \eqref{commut} and by letting $A \to +\infty$, \eqref{local_coerc_virial} is proved. Combining  \eqref{err_vir2}, \eqref{local_morawetz} and \eqref{local_coerc_virial}, we get
\begin{equation}
    \int_0^T\int|x|^{-b}|u|^{\alpha+2} \, dx dt \lesssim R +\frac{T}{R^{\min\{2,b\}}}. 
\end{equation}
By choosing $R = T^\frac{1}{1+{\min\{2,b\}}}$, we finally get
\begin{equation}
    \frac{1}{T}\int_0^T\int|x|^{-b}|u|^{\alpha+2} \, dx dt \lesssim \frac{1}{T^\frac{{\min\{2,b\}}}{1+{\min\{2,b\}}}}. 
\end{equation}
\end{proof}

\begin{proof}[\bf Proof of Theorem \ref{Scattering}]

By Proposition \ref{virial}, there exists a sequence of times $\{t_n\}$ such that $t_n \to +\infty$ and

\begin{equation}
\int |x|^{-b} |u(t_n)|^{2\alpha+2}    \to 0, \text{ as } n \to \infty.
\end{equation}

Now, by choosing $R$ as in Proposition \ref{scattering_criterion} and using H\"older's inequality:
\begin{equation}
\int_{|x|\leq R} |u(t_n)|^2 \lesssim R^\frac{2b+N\alpha}{\alpha+2}\left(\int_{|x|\leq R}|x|^{-b}|u(t_n)|^{\alpha+2}\right)^\frac{2}{\alpha+2} \to 0, \text{ as } n \to +\infty.
\end{equation}
Therefore $u$ scatters forward in time in $H^2(\mathbb{R}^N)$.
\end{proof}

\bibliographystyle{abbrv}

\end{document}